\documentclass[12pt,twoside]{amsart}
\usepackage{amsfonts}
\usepackage{amsmath}
\usepackage{amssymb}
\usepackage{graphicx}
\usepackage{mathrsfs}
\usepackage{multicol}
\usepackage{color}
\usepackage{pst-eucl,pstricks-add}
\usepackage{pst-3dplot}
\usepackage{pst-solides3d}

\usepackage{lscape}

\newtheorem{theo}{Theorem}
\newtheorem{cor}{Corollary}
\newtheorem{lem}{Lemma}
\newtheorem{prop}{Proposition}
\theoremstyle{definition}

\theoremstyle{remark}
\newtheorem{rem}{\bf Remark\/}

\newtheorem{rems}[rem]{\bf Remarks\/}

\numberwithin{equation}{section}

\definecolor{bronze}{rgb}{0.4, 0.7, 0.1}

\def\1{{\mathchoice {\rm 1\mskip-4mu l} {\rm 1\mskip-4mu l}{\rm 1\mskip-4.5mu l} {\rm 1\mskip-5mu l}}}
\newcommand{\ds}{\displaystyle}

\title{Zeros of new Bergman kernels}

\author[N. Ghiloufi]{Noureddine Ghiloufi}
\email{noureddine.ghiloufi@fsg.rnu.tn, nghiloufi@kfu.edu.sa}
\address{Department of mathematics\\College of science\\P.O. box 400 King Faisal University \\  Al-Ahsa, 31982\\ Kingdom of Saudi Arabia.}
\address{University of Gabes\\ Faculty of Sciences of Gabes\\ LR17ES11 Mathematics and Applications\\ 6072, Gabes, Tunisia.}
\author[S. Snoun]{Safa Snoun}
\email{snoun.safa@gmail.com}
\address{University of Gabes\\ Faculty of Sciences of Gabes\\ LR17ES11 Mathematics and Applications\\ 6072, Gabes, Tunisia.}
\begin{document}

\begin{abstract}
    In this paper we determine explicitly the kernels  $\mathbb K_{\alpha,\beta}$ associated with new Bergman spaces $\mathcal A_{\alpha,\beta}^2(\mathbb D)$ considered recently  by the first author and M. Zaway. Then we study the distribution of the zeros of these kernels essentially when $\alpha\in\mathbb N$ where the zeros are given by the zeros of a real polynomial $Q_{\alpha,\beta}$. Some numerical results are given throughout the paper.
\end{abstract}
\maketitle

\section{Introduction}
    The notion of Bergman kernels has several applications and represents an essential tool in complex analysis and geometry. Sometimes it is necessary to  determine explicitly these kernels, however this is not simple in general. In fact if an orthonormal basis of a Hilbert space is given then the Bergman kernel of this space can be obtained as a series using the basis elements. For example, the Bergman kernel of the space $\mathcal A^2_\alpha(\mathbb D)$ of holomorphic functions on the unit disk $\mathbb D$ of $\mathbb C$  that are square integrable with respect to the positive  measure $d\mu_\alpha(z)=(\alpha+1)(1-|z|^2)^\alpha dA(z)$ is given by $\mathbb K_\alpha(z,w)=\frac{1}{(1-z\overline{w})^{\alpha+2}}$.  Hence this kernel has no zero in $\mathbb D$. For more details about Bergman spaces one can see \cite{He-KO-Zh}. In order to obtain kernels with zeros in $\mathbb D$, Krantz consider in his book \cite{Kr} some subspaces of $\mathcal A^2_\alpha(\mathbb D)$. In our statement, instead of considering subspaces, we  modify slightly the measure $d\mu_\alpha$ to obtain a Bergman kernel that is comparable in some sense with the previous one with some zeros in $\mathbb D$. These spaces are considered recently by the first author and Zaway in \cite{Gh-Za}. We recall the main back-ground of this paper:\\

    Throughout the paper, $\mathbb D$ will be the unit disk of the complex plane $\mathbb C$ as it was mentioned before and $\mathbb D^*=\mathbb D\smallsetminus\{0\}$. We set $\mathbb N:=\{0,1,2,\dots\}$ the set of positive integers and $\mathbb R$ the set of real numbers. We claim that a real number $x$ is said to be positive (resp. negative) if $x\geq0$ (resp. $x\leq0$).\\
    For every $-1<\alpha,\beta<+\infty$, we consider the positive measure $\mu_{\alpha,\beta}$ on $\mathbb D$ defined by  $$d\mu_{\alpha,\beta}(z):=\frac{1}{\mathscr B(\alpha+1,\beta+1)}|z|^{2\beta}(1-|z|^2)^\alpha dA(z)$$ where $\mathscr B$ is the beta function defined by $$\mathscr B(s,t)=\int_0^1x^{s-1}(1-x)^{t-1}dx=\frac{\Gamma(s)\Gamma(t)}{\Gamma(s+t)},\quad \forall\; s,t>0$$ and $$dA(z)=\frac{1}{\pi}dx dy=\frac{1}{\pi}rdrd\theta,\quad z=x+iy=re^{i\theta}$$ the normalized area measure on $\mathbb D$.\\
    We denote by $\mathcal A_{\alpha,\beta}^2(\mathbb D)$ the set of holomorphic functions on $\mathbb D^*$ that belongs to the space:
    $$\textbf{L}^2(\mathbb D,d\mu_{\alpha,\beta})=\{f:\mathbb D\longrightarrow \mathbb C;\hbox{ measurable function such that } \|f\|_{\alpha,\beta,2}<+\infty\}$$
    where $$\|f\|_{\alpha,\beta,2}^2:=\int_{\mathbb D}|f(z)|^2d\mu_{\alpha,\beta}(z).$$
    The set $\mathcal A_{\alpha,\beta}^2(\mathbb D)$ is a Hilbert space and $\mathcal A_{\alpha,\beta}^2(\mathbb D)=\mathcal A_{\alpha,m}^2(\mathbb D)$ if $\beta=\beta_0+m$ with $m\in\mathbb N$ and $-1<\beta_0\leq 0$ (see \cite{Gh-Za} for more details). We claim here that $\mathcal A_{\alpha,\beta_0}^2(\mathbb D)=\mathcal A_\alpha^2(\mathbb D)$ is the classical Bergman space equipped with the new norm $\|.\|_{\alpha,\beta_0,2}$. Moreover, for any $\alpha,\beta>-1$, if we set
    \begin{equation}\label{eq1.1}
        e_n(z)=\sqrt{\frac{\mathscr B(\alpha+1,\beta+1)}{\mathscr B(\alpha+1,n+\beta+1)}}\ z^n
    \end{equation}
    for every $n\geq -m$, then the sequence $(e_n)_{n\geq -m}$ is an orthonormal basis of $\mathcal A_{\alpha,\beta}^2(\mathbb D)$. Furthermore, if  $f,g\in\mathcal A_{\alpha,\beta}^2(\mathbb D)$ with $$f(z)=\sum_{n= -m}^{+\infty}a_nz^n,\quad g(z)=\sum_{n= -m}^{+\infty}b_nz^n$$  then $$\langle f,g\rangle_{\alpha,\beta}=\sum_{n= -m}^{+\infty}a_n\overline{b}_n\frac{\mathscr B(\alpha+1,n+\beta+1)}{\mathscr B(\alpha+1,\beta+1)}$$ where $\langle .,.\rangle_{\alpha,\beta}$ is the inner product in $\mathcal A_{\alpha,\beta}^2(\mathbb D)$ inherited from $\textbf{L}^2(\mathbb D,d\mu_{\alpha,\beta})$.\\

    The following main result determine the reproducing kernel of $\mathcal A_{\alpha,\beta}^2(\mathbb D)$.
    \begin{theo}\label{th1}
        Let $-1<\alpha,\beta<+\infty$ and  $\mathbb{K}_{\alpha,\beta}$ be  the reproducing Bergman kernel  of $\mathcal A_{\alpha,\beta}^2(\mathbb D)$. Then
        $$\ds \mathbb{K}_{\alpha,\beta}(w,z)=\frac{Q_{\alpha,\beta}(w\overline{z})}{(w\overline{z})^m(1-w\overline{z})^{2+\alpha}}$$ where
        $$Q_{\alpha,\beta}(\xi)=\left\{
            \begin{array}{lcl}
            (\alpha+1)\mathscr B(\alpha+1,\beta+1)& if &\beta\in\mathbb N\\
            \ds\beta_0\frac{\mathscr B(\alpha+1,\beta+1)}{\mathscr B(\alpha+1,\beta_0+1)}\sum_{n=0}^{+\infty}\frac{(-\xi)^n}{n+\beta_0}{\alpha+1\choose n} &if&\beta\not\in\mathbb N
          \end{array}\right.
        $$
        with $\beta_0=\beta-\lfloor\beta\rfloor-1=\beta-m$.
    \end{theo}

As a consequence of this main result, the study can be reduced to the case $\beta=\beta_0\in]-1,0]$. Indeed if we set  $\mathbb{K}_{\alpha,\beta}(w,z)=\mathcal{K}_{\alpha,\beta}(w\overline{z})$ and
$$\begin{array}{lccl}
      M:&\mathcal A_{\alpha,\beta}^2(\mathbb D)&\longrightarrow& \mathcal A_{\alpha,\beta_0}^2(\mathbb D)\\
      & f& \longmapsto & \ds \frac{\mathscr B(\alpha+1,\beta_0+1)}{\mathscr B(\alpha+1,\beta+1)}z^mf
  \end{array}$$
then the linear operator $M$ is invertible and bi-continuous  and   $\mathcal{K}_{\alpha,\beta}=M^{-1}\circ\mathcal{K}_{\alpha,\beta_0}$. Thus we can assume that $m=0$ i.e. $\beta=\beta_0$.\\

The proof of the main result is the aim of the following section. Then as a consequence, we will prove that for $\alpha\in\mathbb N$ and $\beta\in]-1,0[$, the zeros set of $\mathbb{K}_{\alpha,\beta}$ is a totally real submanifold of $\mathbb D^*\times\mathbb D^*$ with real dimension one  formed by at most $(\alpha+1)$ connected components. This set is reduced to one connected component for $\beta$ closed to $-1$ ($\beta\to (-1)^+$) and it is empty for $\beta$ near $0$  ($\beta\to0^-$). These zeros are related to the zeros set $\mathcal Z_{Q_{\alpha,\beta}}$ of $Q_{\alpha,\beta}$ in $\mathbb C$. Hence we will concentrate essentially on the distribution of $\mathcal Z_{Q_{\alpha,\beta}}$. This will be the aim of the third section of the paper where we start by a general study and we conclude that $\mathcal Z_{Q_{\alpha,\beta}}$ is formed by exactly $(\alpha+1)$ connected regular curves when $\beta$ varies in the interval $]-1,0[$.\\
We finish the paper by some open problems. Using Python software, some numerical results are investigated in the annex of the paper where we confirm numerically some asymptotic results.

\section{Proof of the main result} The proof of the first case is simple (it was done in \cite{Gh-Za}) however, the proof of the second one is more delicate and it will be done by steps. Using the sequence $(e_n)_{n\geq -m}$ given by \eqref{eq1.1}, we deduce that the reproducing kernel of $\mathcal A_{\alpha,\beta}^2(\mathbb D)$ can be written as follows
   $$
    \begin{array}{lcl}
         \mathbb{K}_{\alpha,\beta}(w,z)&=&\ds\sum_{n=-m}^{+\infty}e_n(w)\overline{e_n(z)}=\sum_{n=-m}^{+\infty}\frac{\mathscr B(\alpha+1,\beta+1)}{\mathscr B(\alpha+1,n+\beta+1)}w^n\overline{z}^n\\
          &=&\ds\frac{\mathscr B(\alpha+1,\beta+1)}{(w\overline{z})^m}\sum_{n=0}^{+\infty}\frac{1}{\mathscr B(\alpha+1,n+\beta-m+1)}(w\overline{z})^n\\
          &=&\ds \frac{\mathcal{R}_{\alpha,\beta}(w\overline{z})}{(w\overline{z})^m}=:\mathcal{K}_{\alpha,\beta}(w\overline{z})
      \end{array}$$
      where $$\mathcal{R}_{\alpha,\beta}(\xi)=\mathscr B(\alpha+1,\beta+1)\sum_{n=0}^{+\infty}\frac{\xi^n}{\mathscr B(\alpha+1,n+\beta-m+1)}.$$
 If $\beta=m\in\mathbb N$, then
    $$\begin{array}{lcl}
         \mathcal{R}_{\alpha,m}(\xi)&=&\ds\mathscr B(\alpha+1,m+1)\sum_{n=0}^{+\infty}\frac{\xi^n}{\mathscr B(\alpha+1,n+1)} \\
          &=&\ds\frac{(\alpha+1)\mathscr B(\alpha+1,m+1)}{(1-\xi)^{2+\alpha}} .
      \end{array}$$
We consider now the case $\beta\in]m-1,m[$ with $m\in\mathbb N$ and we prove the result in two steps:\\
    \textbf{$\bullet$ First step: the case $\mathbf{\alpha\in\mathbb N}$}. We start by proving the following preliminary lemma.
    \begin{lem}\label{l1} We have
        $$\ds \mathcal{R}_{\alpha,\beta}(\xi)=\frac{Q_{\alpha,\beta}(\xi)}{(1-\xi)^{2+\alpha}}$$ where $Q_{\alpha,\beta}$ is a polynomial of degree $\alpha+1$ with  $Q_{\alpha,\beta}(1)\neq 0$ that satisfies the recurrence formula:
$$Q_{\alpha+1,\beta}(\xi)=\frac{1}{\alpha+\beta+2}\left[\xi(1-\xi)Q_{\alpha,\beta}'(\xi)+(\alpha+\beta-m+2+(m-\beta)\xi)Q_{\alpha,\beta}(\xi)\right].$$
    \end{lem}

\begin{proof}
If $\alpha=0$, then we have
    $$\begin{array}{lcl}
         \mathcal{R}_{0,\beta}(\xi)&=&\ds\mathscr B(1,\beta+1)\sum_{n=0}^{+\infty}\frac{\xi^n}{\mathscr B(1,n+\beta-m+1)} \\ &=&\ds\frac{1}{\beta+1}\sum_{n=0}^{+\infty}(n+\beta-m+1)\xi^n =\frac{Q_{0,\beta}(\xi)}{(1-\xi)^2}
      \end{array}
    $$
    with $$Q_{0,\beta}(\xi)=\frac{1}{\beta+1}((m-\beta)\xi+\beta-m+1).$$
    Assume that the result is proved for $\alpha\in\mathbb N$ i.e. $$\ds \mathcal{R}_{\alpha,\beta}(\xi)=\frac{Q_{\alpha,\beta}(\xi)}{(1-\xi)^{2+\alpha}}$$ where $Q_{\alpha,\beta}$ is a polynomial of degree $\alpha+1$ with $Q_{\alpha,\beta}(1)\neq0$.
    $$\begin{array}{lcl}
        \mathcal{R}_{\alpha+1,\beta}(\xi)&=&\ds\mathscr B(\alpha+2,\beta+1)\sum_{n=0}^{+\infty}\frac{\xi^n}{\mathscr B(\alpha+2,n+\beta-m+1)}\\
        &=&\ds \mathscr B(\alpha+2,\beta+1)\sum_{n=0}^{+\infty}\frac{\Gamma(\alpha+3+n+\beta-m)}{\Gamma(\alpha+2)\Gamma(n+\beta-m+1)}\xi^n\\
        &=&\ds \frac{\mathscr B(\alpha+1,\beta+1)}{\alpha+\beta+2}\sum_{n=0}^{+\infty}\frac{(\alpha+2+n+\beta-m)}{\mathscr B(\alpha+1,n+\beta-m+1)}\xi^n\\
        &=&\ds\frac{1}{\alpha+\beta+2}\left(\xi\mathcal{R}_{\alpha,\beta}'(\xi)+(\alpha+\beta-m+2)\mathcal{R}_{\alpha,\beta}(\xi)\right)\\
        &=&\ds\frac{1}{\alpha+\beta+2}\left(\xi \frac{Q_{\alpha,\beta}'(\xi)}{(1-\xi)^{2+\alpha}} +\xi \frac{(2+\alpha)Q_{\alpha,\beta}(\xi)}{(1-\xi)^{3+\alpha}}\right.\\
        &&\ds \hfill\left.+(\alpha+\beta-m+2)\frac{Q_{\alpha,\beta}(\xi)}{(1-\xi)^{2+\alpha}}\right)\\
        &=&\ds\frac{Q_{\alpha+1,\beta}(\xi)}{(1-\xi)^{3+\alpha}},
      \end{array}
    $$
      with
      $$Q_{\alpha+1,\beta}(\xi)=\frac{\xi(1-\xi)Q_{\alpha,\beta}'(\xi)+(\alpha+\beta-m+2+(m-\beta)\xi)Q_{\alpha,\beta}(\xi)}{\alpha+\beta+2}.$$

      Thus $Q_{\alpha+1,\beta}$ is a polynomial of degree $\alpha+2$ and $$Q_{\alpha+1,\beta}(1)=\frac{\alpha+2}{\alpha+\beta+2}Q_{\alpha,\beta}(1)\neq0.$$
\end{proof}
Now, we can deduce the proof of Theorem \ref{th1} in the case $\alpha\in\mathbb N$. This will be done by induction on $\alpha$. The result is true for $\alpha=0$. Indeed, we have
$$Q_{0,\beta}(\xi)=\frac{1}{\beta+1}(1+\beta_0-\beta_0\xi)=\beta_0\frac{\mathscr B(1,\beta+1)}{\mathscr B(1,\beta_0+1)} \left(\frac{1}{\beta_0}-\frac{\xi}{1+\beta_0}\right).$$
Assume that the result is true until the value $\alpha$. Thanks to Lemma \ref{l1}, we have
$$\begin{array}{ll}
   & Q_{\alpha+1,\beta}(\xi)=\ds\frac{1}{\alpha+\beta+2}\left(\xi(1-\xi)Q_{\alpha,\beta}'(\xi)+(\alpha+2+\beta_0-\beta_0\xi)Q_{\alpha,\beta}(\xi)\right)\\
=&\ds\frac{\beta_0\mathscr B(\alpha+1,\beta+1)}{(\alpha+\beta+2)\mathscr B(\alpha+1,\beta_0+1)}\left[\sum_{j=1}^{\alpha+1}j\frac{(-1)^{j}}{j+\beta_0}{\alpha+1\choose j}\xi^j\right.\\
& \ds +\sum_{j=1}^{\alpha+2}(j-1)\frac{(-1)^{j+1}}{j-1+\beta_0}{\alpha+1\choose j-1}\xi^j+(\alpha+2+\beta_0)\sum_{j=0}^{\alpha+1}\frac{(-1)^{j}}{j+\beta_0}{\alpha+1\choose j}\xi^j\\
 &\ds \hfill\left.+\beta_0\sum_{j=1}^{\alpha+2}\frac{(-1)^{j}}{j-1+\beta_0}{\alpha+1\choose j-1}\xi^j\right]\\
=&\ds\beta_0\frac{\mathscr B(\alpha+2,\beta+1)}{\mathscr B(\alpha+2,\beta_0+1)}\sum_{j=0}^{\alpha+2}\frac{(-\xi)^{j}}{j+\beta_0}{\alpha+2\choose j}.
  \end{array}
$$

This achieves the first step.\\
\textbf{$\bullet$ Second step: The general case ($\alpha>-1$).} \\
We set
  $$\begin{array}{lcl}
       S_{\alpha,\beta_0}(\xi)&:=&\ds\frac{\mathscr B(\alpha+1,\beta_0+1)}{\beta_0\mathscr B(\alpha+1,\beta+1)}Q_{\alpha,\beta}(\xi)\\
       &=&\ds\frac{(1-\xi)^{\alpha+2}}{\beta_0}\sum_{n=0}^{+\infty}\frac{\mathscr B(\alpha+1,\beta_0+1)}{\mathscr B(\alpha+1,n+\beta_0+1)}\xi^n\\
       &=&\ds \frac{Q_{\alpha,\beta_0}(\xi)}{\beta_0}
    \end{array}
  $$
  and
  $$G_{\alpha,\beta_0}(\xi):=\sum_{n=0}^{+\infty}\frac{(-1)^n}{n+\beta_0}{\alpha+1\choose n}\xi^n.$$
  To prove the result it suffices to attest that $S_{\alpha,\beta_0}=G_{\alpha,\beta_0}$ on $\mathbb D$. To show this equality we will prove that both functions $S_{\alpha,\beta_0}$ and $G_{\alpha,\beta_0}$ satisfy the following differential equation:
    \begin{equation}\label{z1}
        \xi F'(\xi)=-\beta_0 F(\xi)+(1-\xi)^{\alpha+1},\quad \forall\; \xi\in\mathbb D.
    \end{equation}
  It follows that $S_{\alpha,\beta_0}-G_{\alpha,\beta_0}$ satisfies on $\mathbb D^* $ the homogenous differential equation:
    $\xi F'(\xi)=-\beta_0 F(\xi)$. In particular it satisfies the same homogenous differential equation on $]0,1[$. Thus there  exists a constant $\sigma\in\mathbb R$ such that for every $t\in]0,1[$ we have $S_{\alpha,\beta_0}(t)-G_{\alpha,\beta_0}(t)=\sigma t^{-\beta_0}$. Since $S_{\alpha,\beta_0}-G_{\alpha,\beta_0}$ is differentiable at $0$, we get $\sigma=0 $ i.e $S_{\alpha,\beta_0}=G_{\alpha,\beta_0}$ on $]0,1[$ and by the analytic extension principle we conclude that $S_{\alpha,\beta_0}=G_{\alpha,\beta_0}$ on $\mathbb D$.\\
  To finish the proof we will show that both functions $S_{\alpha,\beta_0}$ and $G_{\alpha,\beta_0}$ satisfy the differential equation (\ref{z1}). For $G_{\alpha,\beta_0}$ the result is obvious. Indeed
  $$\begin{array}{lcl}
       \xi G'_{\alpha,\beta_0}(\xi)&=&\ds\sum_{n=0}^{+\infty}\frac{n}{n+\beta_0}{\alpha+1\choose n}(-\xi)^n\\
       &=&\ds\sum_{n=0}^{+\infty}\left(1-\frac{\beta_0}{n+\beta_0}\right){\alpha+1\choose n}(-\xi)^n\\
       &=&\ds (1-\xi)^{\alpha+1}-\beta_0G_{\alpha,\beta_0}(\xi).
    \end{array}
  $$
  Now for $S_{\alpha,\beta_0}$, it is not hard to prove that
  $$\begin{array}{lcl}
       \xi S'_{\alpha,\beta_0}(\xi)&=&\ds-(1-\xi)^{\alpha+1}\sum_{n=0}^{+\infty}\frac{(\alpha+1)\mathscr B(\alpha+1,\beta_0+1)}{(\alpha+1+n+\beta_0)\mathscr B(\alpha+1,n+\beta_0+1)}\xi^n\\
       &=&\ds (1-\xi)^{\alpha+1}-\beta_0S_{\alpha,\beta_0}(\xi).
    \end{array}
  $$
    Thus the proof of Theorem \ref{th1} is finished.\\

    As a first consequence of Theorem \ref{th1}, we obtain the following identity:
    \begin{cor}
        Let $-1<\alpha<+\infty$ and $-1<\beta<0$. For every $n\in\mathbb N$,
        $$\sum_{k=0}^n {\alpha+2\choose k}\frac{(-1)^k}{\mathscr B(\alpha+1,n-k+\beta+1)}=\frac{\beta}{\mathscr B(\alpha+1,\beta+1)}{\alpha+1\choose n}\frac{(-1)^n}{n+\beta}.$$
    \end{cor}
    \begin{proof}
        Thanks to Theorem \ref{th1}, we have
        $$\begin{array}{l}
            S_{\alpha,\beta}(\xi)=\ds \sum_{n=0}^{+\infty}\frac{(-1)^n}{n+\beta}{\alpha+1\choose n}\xi^n\\
            =\ds\frac{\mathscr B(\alpha+1,\beta+1)}{\beta}(1-\xi)^{\alpha+2}\sum_{n=0}^{+\infty}\frac{\xi^n}{\mathscr B(\alpha+1,n+\beta+1)}\\
            =\ds \frac{\mathscr B(\alpha+1,\beta+1)}{\beta}\left[\sum_{n=0}^{+\infty}{\alpha+2\choose n}(-\xi)^n\right]\left[\sum_{n=0}^{+\infty}\frac{\xi^n}{\mathscr B(\alpha+1,n+\beta+1)}\right]\\
            =\ds\frac{\mathscr B(\alpha+1,\beta+1)}{\beta}\sum_{n=0}^{+\infty}d_n\xi^n
        \end{array}$$
        where
        $$d_n=\sum_{k=0}^{n}{\alpha+2\choose k}\frac{(-1)^k}{\mathscr B(\alpha+1,n-k+\beta+1)}.$$
        So the result follows.
    \end{proof}
    Using the proof of Theorem \ref{th1}, one can conclude the following corollary:
    \begin{cor}\label{cor2}
        For every $-1<\alpha$ and $-1<\beta<0$, the function $$G_{\alpha,\beta}(\xi)=\beta^{-1}Q_{\alpha,\beta}(\xi)=\sum_{n=0}^{+\infty}\frac{(-1)^n}{n+\beta}{\alpha+1\choose n}\xi^n$$ satisfies:
        $$\xi G'_{\alpha,\beta}(\xi)=\ds (1-\xi)^{\alpha+1}-\beta G_{\alpha,\beta}(\xi)$$
        and
         $$\begin{array}{lcl}
             G_{\alpha+1,\beta}(\xi)&=&\ds\frac{1}{\alpha+\beta+2}\left(\xi(1-\xi)G_{\alpha,\beta}'(\xi)+(\alpha+\beta+2-\beta\xi) G_{\alpha,\beta}(\xi)\right)\\
         &=&\ds\frac{1}{\alpha+\beta+2}\left((\alpha+2)G_{\alpha,\beta}(\xi)+(1-\xi)^{\alpha+2}\right).
           \end{array}$$
    \end{cor}
    \begin{rems} $ $
    \begin{enumerate}
        \item Using the Stirling formula, one can prove that $G_{\alpha,\beta}$ is bounded  on the closed unit disk $\overline{\mathbb D}$. This fact will be used frequently in the hole of the paper.
      \item Thanks to Lemma \ref{l1}, for $\alpha\in\mathbb N$, one has $G_{\alpha,\beta}(1)\neq 0$. For the general case, if $G_{\alpha_0,\beta}(1)\neq 0$ for some $-1<\alpha_0\leq0$ then $G_{\alpha_0+n,\beta}(1)\neq 0$ for every $n\in\mathbb N$.
    \end{enumerate}
  \end{rems}
           In the rest of the paper, we assume that $G_{\alpha,\beta}(1)\neq 0$. This may be true for any $-1<\alpha$ and $-1<\beta<0$.
\section{Zeros of Bergman kernels}
    Using Theorem \ref{th1}, the function $\mathcal K_{\alpha,0}$ has no zero in the unit disk $\mathbb D$. However if $-1<\beta<0$ then $\mathcal K_{\alpha,\beta}$ may have some zeros in $\mathbb D$. We claim that if $\xi\in\mathbb D^*$ is a zero of $\mathcal K_{\alpha,\beta}$ then the sets $\{(z,w)\in\mathbb D^2;\ w\overline{z}=\xi\}$ and $\{(z,w)\in\mathbb D^2;\ z\overline{w}=\xi\}$ define two totally real algebraic surfaces (of real dimension equal to $2$) of $\mathbb C^2$ that are contained in the zeros set of the Bergman kernel $\mathbb K_{\alpha,\beta}$. Thus it suffices to study the zeros set of $\mathcal K_{\alpha,\beta}$.\\
    Due to an algebraic problem, we focus sometimes on the case $\alpha\in\mathbb N$, because in this case the zeros of $\mathcal K_{\alpha,\beta}$ are given by the zeros of the polynomial $G_{\alpha,\beta}$ contained in $\mathbb D$. Thus for $\alpha\in\mathbb N$, we will study the zeros set of $G_{\alpha,\beta}$ in the hole complex plane $\mathbb C$. It is interesting to discuss the variations of these sets in terms of the parameter $\beta$. All results on $G_{\alpha,\beta}$ can be viewed as particular cases of those of the following linear transformation.
\subsection{The Linear transformation $T_\beta$}
    If $\mathcal O(\mathbb D(0,R))$ is the space of holomorphic function  on the disk $\mathbb D(0,R)$ and $-1<\beta<0$, then we define $T_\beta$  on $\mathcal O(\mathbb D(0,R))$ by $$T_\beta f(z)=\sum_{n=0}^{+\infty}\frac{a_n}{n+\beta}z^n$$ for any $f(z)=\sum_{n=0}^{+\infty}a_nz^n$. The transformation $T_\beta$ is linear and bijective from $\mathcal O(\mathbb D(0,R))$ onto itself. It transforms any polynomial to a polynomial with the same degree. We start by the study of zeros of $T_\beta f$ in general then we specialize the study to the case $f(z)=P_\alpha(z)=(1-z)^{\alpha+1}$ where $T_\beta P_\alpha$ is exactly $G_{\alpha,\beta}$.

    \begin{theo}\label{th2}
    Let $0<R\leq+\infty$ and $f$ be a holomorphic function on $\mathbb D(0,R)$ such that $(f(0),f'(0))\neq(0,0)$. Then for every $0<r_0<R$, there exist $$-1<\beta_1(f,r_0)\leq -\frac{|f(0)|}{|f(0)|+r_0|f'(0)|}\leq \beta_2(f,r_0)<0$$ that depend on $f$ and $r_0$ such that the function $T_\beta f$ has no zero in $\mathbb D(0,r_0)$ for every $\beta_2(f,r_0)<\beta<0$ and has exactly one simple zero in $\mathbb D(0,r_0)$ for every $-1<\beta<\beta_1(f,r_0)$.
\end{theo}
 When $f(0)=0$ and $f'(0)\neq0$ the result is reduced to " \emph{$0$ is the unique zero (simple) of the function $T_\beta f$ in $\mathbb D(0,r_0)$ for every $-1<\beta<0$} ". However, when $f'(0)=0$ and $f(0)\neq0$ then " \emph{the function $T_\beta f$ has no zero in $\mathbb D(0,r_0)$ for every $-1<\beta<0$}. "
\begin{proof}
If $f(z)=\sum_{n=0}^{+\infty}a_nz^n$ for every $z\in\mathbb D(0,R)$ with $(a_0,a_1)\neq (0,0)$ then we  set
    $$F_{\beta,f}(z)=\frac{a_0}{\beta}+\frac{a_1}{1+\beta}z.$$
    If $|z|=r_0$ we have
    $$\left|T_\beta f(z)-F_{\beta,f}(z)\right|\leq \sum_{n=2}^{+\infty} \frac{|a_n|}{n+\beta}r_0^n.$$
    Moreover, if we set
    $$\psi(\beta)=\left|\frac{|a_0|}{\beta}+\frac{|a_1|r_0}{1+\beta}\right|-\sum_{n=2}^{+\infty} \frac{|a_n|}{n+\beta}r_0^n$$
    then $$\psi\left(-\frac{|a_0|}{|a_0|+r_0|a_1|}\right)<0\  and \lim_{\beta\to0^-}\psi(\beta)=+\infty\ (resp.\ \lim_{\beta\to(-1)^+}\psi(\beta)=+\infty)$$
    when $a_0\neq0$ (resp. $a_1\neq0$). It follows that there exist $$-1<\beta_1\leq -\frac{|a_0|}{|a_0|+r_0|a_1|}\leq \beta_2<0$$ that depend on $f$ and $r_0$ such that for every $\beta\in]-1,\beta_1[\cup]\beta_2,0[$ one has $\psi(\beta)>0$.  Hence, for every $\beta\in]-1,\beta_1[\cup]\beta_2,0[$ and $|z|=r_0$, we have $|T_\beta f(z)-F_{\beta,f}(z)|<|F_{\beta,f}(z)|.$
    Thus by Rouch\'e Theorem, $T_\beta f$ and $F_{\beta,f}$ have the same number of zeros counted with their multiplicities in the disk $\mathbb D(0,r_0)$.
\end{proof}
    In the following lemma we collect some useful properties of  $T_\beta f$ that will be used frequently in the sequel.
    \begin{lem}
    If $f$ is a holomorphic function on $\mathbb D(0,R)$ and $-1<\beta<0$ then the following  assertions hold
    \begin{enumerate}
      \item The number $0$ is a zero of $f$ if and only if it is a zero of $T_\beta f$ (with the same multiplicity).
      \item The derivative of $T_\beta f$ satisfies $$z(T_\beta f)'(z)=f(z)-\beta T_\beta f(z),\quad \forall\; z\in\mathbb D(0,R).$$
      \item The functions $f$ and $T_\beta f$ have a commun zero in $\mathbb D^*(0,R)$  if and only if $f\equiv0$.
      \item The function $T_\beta f$ has a zero in  $\mathbb D^*(0,R)$ with multiplicity greater or equal to $2$ if and only if $f\equiv0$.
    \end{enumerate}
    \end{lem}

    Now we consider a fixed holomorphic function $f$ on $\mathbb D(0,R)$ with $f(0)\neq 0$. We set
    $$H_f(\beta,z):=T_\beta f(z)=\sum_{n=0}^{+\infty}\frac{a_n}{n+\beta}z^n$$   for $(\beta,z)\in]-1,0[\times \mathbb D(0,R)$ and $$\mathscr D_f:=\{(\beta,z)\in]-1,0[\times \mathbb D(0,R);\ H_f(\beta,z)=0\}.$$
    We assume that the set $\mathscr D_f$ is not empty. Indeed if $f\equiv c$ is a constant function then $T_\beta f\equiv \frac{c}{\beta}$, thus $\mathscr D_c=\emptyset$ if $c\neq0$ and $\mathscr D_c=\mathbb C$ if $c=0$. Moreover it is easy to find some examples of non constant holomorphic functions $g$ where $T_\beta g$  has no zero for some value of $\beta$. But we don't know if there exists a non constant function $g$ such that $\mathscr D_g$ is empty.
    \begin{prop}
        The set $\mathscr D_f$ is a submanifold of (real) dimension one in $\mathbb R^3$  formed by at most countable connected components $(\mathcal Y_{f,k})_k$.\\
        If $\mathcal Y$ is a connected component of $\mathscr D_f$  then there exist $-1\leq a_{\mathcal Y}<b_{\mathcal Y}\leq0$ and  a $\mathcal C^\infty-$function $\mathcal X:]a_{\mathcal Y},b_{\mathcal Y}[\longrightarrow \mathbb D(0,R)$ such that $$\ds \mathcal Y=Graph(\mathcal X):=\{(\beta,\mathcal X(\beta));\ \beta\in ]a_{\mathcal Y},b_{\mathcal Y}[\}.$$
        Moreover for every $\beta\in ]a_{\mathcal Y},b_{\mathcal Y}[$, one has
        \begin{equation}\label{eq3.4}
            \mathcal X'(\beta)=\frac{\mathcal X(\beta)}{f(\mathcal X(\beta))}\sum_{n=0}^{+\infty}  \frac{a_n}{(n+\beta)^2}(\mathcal X(\beta))^n.
        \end{equation}
    \end{prop}
    \begin{proof}
        For every $(\beta,z)\in\mathscr D_\alpha$ we have
        $$\frac{\partial H_\alpha}{\partial z}(\beta,z)=(T_\beta f)'(z)=\frac{1}{z}f(z)\neq0.$$
        The result follows using the implicit functions theorem.
    \end{proof}
    It is easy to see that if $0<R<+\infty$ then $a_{\mathcal Y}>-1$ for  all connected components $\mathcal Y$ of $\mathscr D_f$ except the unique component $\mathcal Y_{f,0}$ given by Theorem \ref{th2} where $a_{\mathcal Y_{f,0}}=-1$. However,  $b_{\mathcal Y}=0$ if and only if $R=+\infty$ i.e. $f$ is an entire function. In this case, of entire functions,  all functions $\mathcal X_{f,k}$ are defined on $]-1,0[$.
    \begin{rem}
    Using the same proof, the previous result can be improved to the complex case  as follows:\\
    If we set $\Omega:=\{\beta\in\mathbb C;\ -1<\Re e(\beta)<0\}$ and $$\mathcal D_f:=\{(\beta,z)\in\Omega\times \mathbb D(0,R);\ H_f(\beta,z)=0\}$$ then $\mathcal D_f$ is a submanifold of (complex) dimension  one in $\Omega\times\mathbb D(0,R)$ formed by connected components. Thus, the Lelong number of the current $[\mathcal D_f]$ of integration over $\mathcal D_f$ is equal to one at every point of $\mathcal D_f$. (This is due to the fact that all zeros of $H_f$ are simple).
    \end{rem}
     The following theorem gives the asymptotic behaviors of functions  $\mathcal X_f$ near $-1$ and $0$ when $f$ is a polynomial. We claim that if $f(z)=a_0+a_1z$ then the solution is explicitly determined by $$\mathcal X_f(\beta)=-\frac{a_0}{a_1}\frac{\beta+1}{\beta}.$$
     Hence we will consider the case when $deg(f)\geq2$.
    \begin{theo}\label{th3}
        Let $f(z)=\sum_{n=0}^pa_nz^n$ be a polynomial of degree $p\geq 2$ with $f(0)\neq0$. We set $a_n=|a_n|e^{i\theta_n}$ for any $0\leq n\leq p$. The set $\mathscr D_f$ is formed exactly by $p$ connected components $(\mathcal Y_{f,k})_{0\leq k\leq p-1}$ with the corresponding functions $\mathcal X_{f,k}:]-1,0[\longrightarrow \mathbb C$. Again we keep $\mathcal X_{f,0}$ to indicate the function related to the unique component given by Theorem \ref{th2}.
        \begin{enumerate}
          \item For every $0\leq k\leq p-1$, we have $\lim_{\beta\to 0^-}|\mathcal X_{f,k}(\beta)|=+\infty$ and
            \begin{equation}\label{eq3.5}
                \mathcal X_{f,k}(\beta)\underset{\beta\to 0^-}\sim \left(-\frac{p|a_0|}{\beta|a_p|}\right)^{\frac{1}{p}}e^{i\frac{\theta_0-\theta_p+2j_k\pi}{p}}
            \end{equation}
            where $j_k\in\mathbb Z$ that depends on $k$.
          \item If $f'(0)\neq 0$ then for every $1\leq k\leq p-1$
            $$\lim_{\beta\to (-1)^+}|\mathcal X_{f,k}(\beta)|=+\infty\quad  and \quad \lim_{\beta\to (-1)^+}\mathcal X_{f,0}(\beta)=0.$$
            Moreover, we have
            \begin{equation}\label{eq3.7}
                \mathcal X_{f,0}(\beta)\underset{\beta\to (-1)^+}\sim \frac{a_0}{a_1}(1+\beta).
            \end{equation}
            and
            \begin{equation}\label{eq3.6}
                \mathcal X_{f,k}(\beta)\underset{\beta\to(-1)^+}\sim \left(\frac{(p-1)|a_1|}{(1+\beta)|a_p|}\right)^{\frac{1}{p-1}}e^{\frac{i(\theta_1-\theta_p+(2s_k+1)\pi)}{p-1}},  \quad \forall \; 1\leq k\leq p-1
            \end{equation}
            for some $s_k\in\mathbb Z$ that depends on $k$.
        \end{enumerate}
\end{theo}
If $f(0)\neq 0$ and $f'(0)=0$ then all functions $\mathcal X_{f,k}$ are bounded near $-1$.
\begin{proof}
    Let $0\leq k\leq p-1 $.  As $a_0\neq0$ then using the equality
        $$\frac{a_0}{\beta}+\sum_{n=1}^{p}\frac{a_n}{n+\beta}(\mathcal X_{f,k}(\beta))^n=0$$
    we obtain
        $$\lim_{\beta\to 0^-}\mathcal X_{f,k}(\beta)=\infty$$
    and
        $$-\frac{a_0}{\beta}\underset{\beta\to 0^-}\sim \frac{a_p}{p+\beta}(\mathcal X_{f,k}(\beta))^p.$$
    That means
        $$(\mathcal X_{f,k}(\beta))^p\underset{\beta\to 0^-}\sim     -\frac{pa_0}{a_p\beta}$$
    so we get Equation \eqref{eq3.5}.\\

    With the same way if $a_1\neq0$ then for every $0\leq k\leq p-1 $ we have $$\lim_{\beta\to (-1)^+}\mathcal X_{f,k}(\beta)\in\{0,\infty\}.$$ Thanks to Theorem \ref{th2}, $$\lim_{\beta\to (-1)^+}\mathcal X_{f,0}(\beta)=0\quad and\quad \lim_{\beta\to (-1)^+}\mathcal X_{f,k}(\beta)=\infty,\quad \forall\; 1\leq k\leq p-1.$$
    Thus, we obtain
        $$\mathcal X_{f,0}(\beta)\underset{\beta\to (-1)^+}\sim-\frac{a_0}{a_1}\frac{1+\beta}{\beta}.$$
    Therefore, Equation \eqref{eq3.7} follows.\\
    For $1\leq k\leq p-1 $ we obtain
        $$\frac{a_1}{1+\beta}\underset{\beta\to (-1)^+}\sim -\frac{a_p}{p+\beta}(\mathcal X_{f,k}(\beta))^{p-1}$$
    thus,
        $$(\mathcal X_{f,k}(\beta))^{p-1}\underset{\beta\to (-1)^+}\sim -\frac{(p-1)a_1}{a_p(1+\beta)}$$
    and Equation \eqref{eq3.6} follows.
\end{proof}
\subsection{Application on the Bergman kernels}
As mentioned before, for any $\alpha\in\mathbb N$, $T_\beta P_\alpha=G_{\alpha,\beta}$ where $P_\alpha(z)= (1-z)^{\alpha+1}$. Hence all previous results are valid and more precisions are needed to accomplish the study of $ \mathcal X_{\alpha,k}:=\mathcal X_{P_\alpha,k},\ 0\leq k\leq \alpha$. We start by claiming that if $x<0$ then there exists $\beta_x\in]-1,0[$ such that $G_{\alpha,\beta_x}(x)=0$. It follows that $(\beta_x,x)$ is in a component (says $\mathcal Y_{\alpha,0}$) of $\mathscr D_\alpha:=\mathscr D_{P_\alpha}$. Hence, the corresponding  function $\mathcal X_{\alpha,0}$ maps $]-1,0[$ onto $]-\infty,0[$. Indeed we have $\mathcal X_{\alpha,0}'(\beta)<0$ for every $\beta\in]-1,0[$ thus it is a decreasing function and
$$\lim_{\beta\to0^-} \mathcal X_{\alpha,0}(\beta)=-\infty,\quad \lim_{\beta\to(-1)^+} \mathcal X_{\alpha,0}(r)=0.$$ Using Corollary \ref{cor2}, we can deduce that $\mathcal X_{\alpha,0}(\beta)\geq \mathcal X_{\alpha+1,0}(\beta)$ for every  $\beta\in]-1,0[$ (See Figure \ref{fig1}).
\begin{rem}
    For every $\alpha\in\mathbb N$, we set $-1<s_\alpha<0$ the unique  solution of $\mathcal X_{\alpha,0}(\beta)=-1$. The polynomial $G_{\alpha,\beta}$ has no zero in $]-1,0[$ for every $s_\alpha<\beta<0$ and has exactly one simple zero in $]-1,0[$ for every $-1<\beta<s_\alpha$.
\end{rem}
We claim that $(s_\alpha)_\alpha$ is an increasing  sequence (See again Figure \ref{fig1}).\\

\begin{figure}[h]
  \includegraphics[scale=0.2]{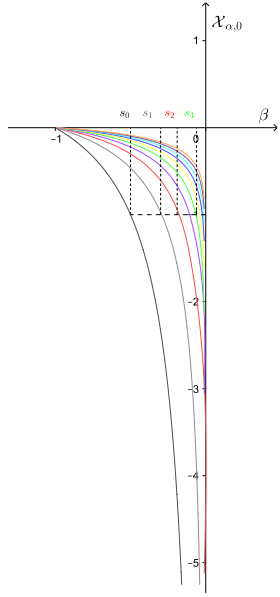}\\
  \caption{Graphs of $\mathcal X_{\alpha,0}$ for $0\leq \alpha\leq 9$.}\label{fig1}
\end{figure}

The following lemma explain differently the conclusion of Theorem \ref{th2} in the current statement (See Table \ref{t1} for numerical values of $\beta_1$ and $\beta_2$ given by Theorem \ref{th2} for this example).
\begin{lem}\label{l3}
    For every $\alpha>-1$, the family of functions $(\beta(1+\beta)G_{\alpha,\beta})_{\beta\in]-1,0[}$ converges uniformly on $\mathbb D$ to the constant $1$ (resp. to the polynomial $(\alpha+1)\xi$) as $\beta\to 0^-$ (resp. as $\beta\to (-1)^+$).\\
    In particular, for every $m\in\mathbb N$ (resp. $m\in\mathbb N^*$) the family of kernels $(\mathcal K_{\alpha,\beta})_ {\beta\in]m-1,m[}$ converges uniformly on every compact subset of $\mathbb D^*$ to $\mathcal K_{\alpha,m}$ (resp. to $\mathcal K_{\alpha,m-1}$) as $\beta\to m^-$ (resp. as $\beta\to (m-1)^+$).
\end{lem}
\begin{proof}
    The lemma is a simple consequence of the following equality:
    $$\beta(1+\beta)G_{\alpha,\beta}(\xi)=(1+\beta)-\beta(1+\alpha)\xi+\beta(1+\beta)\sum_{n=2}^{+\infty}\frac{(-\xi)^n}{n+\beta}{\alpha+1\choose n}$$
    and the fact that the series converges normally on $\mathbb D$ (obtained using the Stirling formula).
\end{proof}

\begin{table}[h]
\caption{Numerical values of $\beta_1(P_\alpha,1)$ and $\beta_2(P_\alpha,1)$ given by Theorem \ref{th2}.}\label{t1}
\begin{tabular}{lcc}
  \hline
  $\alpha$& $\beta_1(P_\alpha,1)$& $\beta_2(P_\alpha,1)$\\
  \hline
  2&$-0.381966$&$-0.177124$ \\
  3&$-0.493058$ & $-0.107610$\\
  4&$ -0.667086$& $-0.0649539$  \\
  5& $-0.793482$& $-0.0387481$ \\
  6 & $-0.870294$& $-0.0227925$\\
  7&$-0.917737$ &$-0.0132128$ \\
  8 & $-0.947843$&$-0.00755239$\\
  9&$-0.967185$&$-0.0042614$  \\
  \hline
\end{tabular}
\end{table}

The most important conclusion of this lemma is the continuity of the Bergman kernels $\mathbb K_{\alpha,\beta}$ in terms of the parameter $\beta$. Essentially the fact that the Bergman kernel $\mathbb K_{\alpha,\beta}$ converges uniformly on every compact subset of $\mathbb D^2$ to the classical Bergman kernel  $\mathbb K_{\alpha,0}$ when $\beta\to 0^-$.\\

Now we will focus on the other components of $\mathscr D_\alpha$. We use $\mathcal X_{\alpha,k},\ 0\leq k\leq \alpha$ to indicate the corresponding functions such that $\Im m(\mathcal X_{\alpha,k}(\beta))\leq 0$  for every $0\leq k\leq \lfloor\frac{\alpha+1}{2}\rfloor$ and $\mathcal X_{\alpha,\alpha+1-k}(\beta)=\overline{\mathcal X}_{\alpha,k}(\beta)$ for every $1\leq k\leq \alpha$. Theorem \ref{th3} can be written as follows:
\begin{prop}\label{prop3}
For every $\alpha\in\mathbb N$, we have
$$\left\{\begin{array}{lcl}
           \mathcal X_{\alpha,k}(\beta)&\underset{\beta\to 0^-}\sim&\ds \left(-\frac{\alpha+1}{\beta}\right)^{\frac{1}{\alpha+1}} e^{i\frac{(2k-\alpha-1)\pi}{\alpha+1}},  \quad \forall\; 0\leq k\leq\alpha\\
           \mathcal X_{\alpha,k}(\beta)&\underset{\beta\to(-1)^+}\sim &\ds \left(\frac{\alpha(\alpha+1)}{1+\beta}\right)^{\frac{1}{\alpha}}e^{\frac{i(2k-\alpha-1)\pi}{\alpha }},  \quad \forall \; 1\leq k\leq\alpha\\
           \mathcal X_{\alpha,0}(\beta)&\underset{\beta\to (-1)^+}\sim&\ds -\frac{1+\beta}{\alpha+1}.
         \end{array}\right.$$
\end{prop}
The following figures (figures \ref{fig2} and \ref{fig3}) explain numerically the result of Proposition \ref{prop3}.
\begin{figure}[h]
  \includegraphics[scale=0.3]{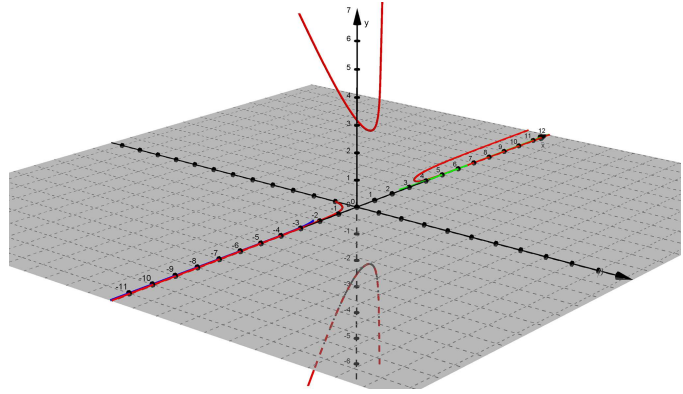}\\
  \caption{Graphs of ${\mathcal{X}_{3,\bullet}}$ (in red) with asymptotic curves ( to $\mathscr{C}_{\mathcal{X}_{3,0}}$ in blue  and to $\mathscr{C}_{\mathcal{X}_{3,2}}$ in green)}\label{fig2}
\end{figure}

\begin{figure}
  \includegraphics[scale=0.5]{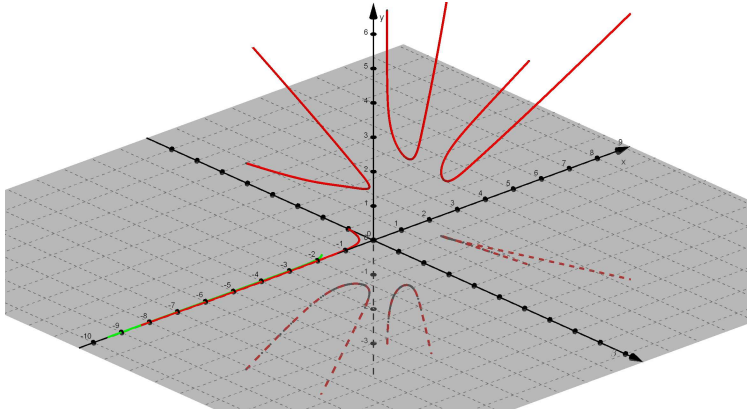}\\
  \caption{Graphs of ${\mathcal{X}_{6,\bullet}}$ (in red) with asymptotic curve to $\mathscr{C}_{\mathcal{X}_{6,0}}$ (in green)}\label{fig3}
\end{figure}
\subsection{Even and odd Bergman kernels}
Following the idea of Krantz developed in \cite{Kr}, we consider the subspaces $\mathscr E_{\alpha,\beta}^2(\mathbb D)$ and $\mathscr L_{\alpha,\beta}^2(\mathbb D)$ of $\mathcal A_{\alpha,\beta}^2(\mathbb D)$ generated respectively by the even $(e_{2n})_{n\geq0}$ and the odd $(e_{2n+1})_{n\geq0}$ sequences. Hence  $\mathscr E_{\alpha,\beta}^2(\mathbb D)$ and $\mathscr L_{\alpha,\beta}^2(\mathbb D)$ are Hilbert subspaces of $\mathcal A_{\alpha,\beta}^2(\mathbb D)$ formed respectively by even and   odd functions. The reproducing Bergman kernels of these spaces are given by $\mathbb E_{\alpha,\beta}(z,w)=\mathcal E_{\alpha,\beta}(z\overline{w})$ and $\mathbb L_{\alpha,\beta}(z,w)=\mathcal L_{\alpha,\beta}(z\overline{w})$ where
$$\begin{array}{lcl}
     \mathcal E_{\alpha,\beta}(\xi)&=&\frac{1}{2}(\mathcal K_{\alpha,\beta}(\xi)+\mathcal K_{\alpha,\beta}(-\xi))\\
     &=&\ds \frac{1}{2(1-\xi^2)^{\alpha+2}}\left((1+\xi)^{\alpha+2}Q_{\alpha,\beta}(\xi)+(1-\xi)^{\alpha+2}Q_{\alpha,\beta}(-\xi)\right)\\
     &=:&\ds \frac{\mathcal I_{\alpha,\beta}(\xi)}{2(1-\xi^2)^{\alpha+2}}
  \end{array}$$
and
$$\begin{array}{lcl}
     \mathcal L_{\alpha,\beta}(\xi)&=&\frac{1}{2}(\mathcal K_{\alpha,\beta}(\xi)-\mathcal K_{\alpha,\beta}(-\xi))\\
     &=&\ds \frac{1}{2(1-\xi^2)^{\alpha+2}}\left((1+\xi)^{\alpha+2}Q_{\alpha,\beta}(\xi)-(1-\xi)^{\alpha+2}Q_{\alpha,\beta}(-\xi)\right)\\
     &=:&\ds \frac{\mathcal J_{\alpha,\beta}(\xi)}{2(1-\xi^2)^{\alpha+2}}
  \end{array}$$

    Again, to study the zeros of even and odd Bergman kernels, it suffices to study the zeros of the corresponding functions $\mathcal I_{\alpha,\beta}$ and $\mathcal J_{\alpha,\beta}$. Let $\varepsilon_{\alpha,\beta}$ (resp. $\Theta_{\alpha,\beta}$) be the number of zeros of the function $\mathcal I_{\alpha,\beta}$ (resp. $\mathcal J_{\alpha,\beta}$) in the unit disk $\mathbb D$ counted with their multiplicities. To determine $\varepsilon_{\alpha,\beta}$ and $\Theta_{\alpha,\beta}$ in the case when $\alpha\in\mathbb N$, we start by the case $\beta=0$. In this case it is easy  to check that the zeros of $\mathcal I_{\alpha,0}$ are given by $z_k:=-i\tan\left(\frac{(2k+1)\pi}{2(\alpha+2)}\right)$ where $0\leq k\leq \alpha+1$ (we omit the value $k$ for which $\cos(\frac{(2k+1)\pi}{2(\alpha+2)})=0$ whenever $\alpha$ is odd). Similarly to the even case, the zeros  of $\mathcal J_{\alpha,0}$ are given by $w_k:=-i\tan\left(\frac{k\pi}{\alpha+2}\right),\ 0\leq k\leq \alpha+1$. It follows that if $\alpha=4\tau+r$ with $\tau\in\mathbb N$ and $0\leq r\leq 3$ then
        $$\varepsilon_{\alpha,0}=\left\{\begin{array}{lcl}
                             2\tau& if& r=0\\
                             2\tau+2 & if& 1\leq r\leq 3
                          \end{array}\right.
        \qquad \Theta_{\alpha,0}=\left\{\begin{array}{lcl}
                             2\tau+1& if& 0\leq r\leq 2\\
                             2\tau+3 & if& r=3
                          \end{array}\right.
        $$
    \begin{prop}
        Let $\alpha=4\tau+r\in\mathbb N$ with $\tau\in\mathbb N$ and $0\leq r\leq 3$.
        \begin{enumerate}
            \item If $r\neq0$ then:
                \begin{enumerate}
                    \item  There exists $-1<\beta_4<0$ such that for every $\beta_4<\beta\leq0$ we have $\varepsilon_{\alpha,\beta}=\varepsilon_{\alpha,0}$.
                    \item There exists $-1<\beta_5<0$ such that for every $-1<\beta<\beta_5$ we have $\Theta_{\alpha,\beta}=\varepsilon_{\alpha,0}+1$.
                \end{enumerate}
          \item If $r\neq2$ then:
                \begin{enumerate}
                  \item There exists $-1<\beta_3<0$ such that for every $-1<\beta<\beta_3$ we have $\varepsilon_{\alpha,\beta}=\Theta_{\alpha,0}+1$.
                  \item There exists $-1<\beta_6<0$ such that for every $\beta_6<\beta\leq0$ we have $\Theta_{\alpha,\beta}=\Theta_{\alpha,0}$.
                \end{enumerate}
        \end{enumerate}
    \end{prop}

    \begin{proof}
        We claim that $\pm i$ are zeros of $\mathcal I_{\alpha,0}$ (resp. $\mathcal J_{\alpha,0}$) when $\alpha=4\tau$ (resp. $\alpha=4\tau+2$). For this reason we omit the corresponding values of $\alpha$ in the proposition in order to use the Rouch\'e theorem. Hence it suffices to study the convergence in terms of the parameter $\beta$.\\ Thanks to Lemma \ref{l3}, the family of polynomials $((1+\beta)Q_{\alpha,\beta}(\xi))_{-1<\beta<0}$ converges to  $1$ on $\mathbb D$ when $\beta\to 0^-$ and to the polynomial $(\alpha+1)\xi$ when $\beta\to (-1)^+$. It follows that $(1+\beta)\mathcal I_{\alpha,\beta}(\xi)$ converges to $\mathcal I_{\alpha,0}(\xi)$ on $\mathbb D$ as $\beta\to 0^-$ and to $(\alpha+1)\xi \mathcal J_{\alpha,0}(\xi)$ on $\mathbb D$ as $\beta\to (-1)^+$.\\
        For the odd case, the family $(1+\beta)\mathcal J_{\alpha,\beta}(\xi)$ converges to $\mathcal J_{\alpha,0}(\xi)$ when $\beta\to 0^-$ and to $(\alpha+1)\xi \mathcal I_{\alpha,0}(\xi)$ when $\beta\to (-1)^+$. Using the Rouch\'e theorem the result follows.
    \end{proof}
    To improve the previous result, we consider the number of zeros $\widehat{\varepsilon}_{\alpha,0}$ (resp. $\widehat{\Theta}_{\alpha,0}$) of the function $\mathcal I_{\alpha,0}$
    (resp. $\mathcal J_{\alpha,0}$) in the closed unit disk $\overline{\mathbb D}$ given by $\widehat{\varepsilon}_{\alpha,0}=2\tau+2$ and
        $$\widehat{\Theta}_{\alpha,0}=\left\{\begin{array}{lcl}
                             2\tau+1& if& 0\leq r\leq 1\\
                             2\tau+3 & if& 2\leq r\leq 3
                          \end{array}\right.
        $$ where $\alpha=4\tau+r$. Using the same idea, one can prove the following corollary:
    \begin{cor}\label{cor3}
        Let $\alpha\in\mathbb N$ and $\eta_0=\tan\left(\frac{\pi}{4}+\frac{\pi}{\alpha+2}\right)$.
        \begin{enumerate}
          \item There exist $-1<\beta_3<\beta_4<0$ that depend on $\alpha$ such that for every $1<\eta<\eta_0$, the polynomial $\mathcal I_{\alpha,\beta}(\eta\xi)$ has exactly $\widehat{\varepsilon}_{\alpha,0}$ zeros in $\mathbb D$ for every $\beta\in]\beta_4,0]$ and $\widehat{\Theta}_{\alpha,0}+1$ zeros in $\mathbb D$ for every $\beta\in]-1,\beta_3[$.
          \item There exist $-1<\beta_5<\beta_6<0$ that depend on $\alpha$ such that for every $1<\eta<\eta_0$, the polynomial $\mathcal J_{\alpha,\beta}(\eta\xi)$ has exactly $\widehat{\Theta}_{\alpha,0}$ zeros in $\mathbb D$ for every $\beta\in]\beta_6,0]$ and $\widehat{\varepsilon}_{\alpha,0}+1$ zeros in $\mathbb D$ for every $\beta\in]-1,\beta_5[$.
        \end{enumerate}
    \end{cor}
    If the conditions of the previous proposition are satisfied then one can take $\eta=1$ in the corollary to obtain the same result given by the proposition.

\section{Open problems}
It is interesting to study the asymptotic distribution of zeros of $G_{\alpha,\beta}$ when $\alpha\in\mathbb N$ and goes to infinity. In other words, can we find a positive measure  $\mu$ such that the sequence of measures $$\mu_{\alpha,\beta}:=\frac{1}{\alpha+1}\sum_{j=0}^{\alpha+1}\delta_{\mathcal X_{\alpha,j}(\beta)}$$ converges weakly to the measure $\mu$ as $\alpha\to+\infty$? Geometrically, the distribution of the set $\{\mathcal X_{\alpha,j}(\beta),\ 0\leq j\leq \alpha+1\}$ may depend on $\alpha$ in some non trivial way. For example, the distribution of sets   $\mathcal X_{\alpha,\bullet}(-10^{-4})$ for $\alpha=31,32,33,34,36$  are similar (see figure \ref{fig4}) however these are different to the one that correspond to $\alpha=35$ (see figure \ref{fig5}).\\
Can we find explicitly the equation of  the parametric curve that describe  the set $\mathcal X_{\alpha,\bullet}(\beta)$? (this curve may be a circle in figure \ref{fig4} for   $\alpha=31,32,33,34,36$). See also figures \ref{fig5} and \ref{fig6}.\\

It is simple to prove that for every $1\leq k\leq\alpha$, there exists $t_{\alpha,k}\in]-1,0[$ such that $$|\mathcal X_{\alpha,k}(t_{\alpha,k})|=\min_{-1<\beta<0}|\mathcal X_{\alpha,k}(\beta)|$$
    and satisfies
    $$ \sum_{j,k=0}^{\alpha+1} {\alpha+1\choose{j}}{\alpha+1\choose{k}} \frac{(-1)^{j+k}}{(j+t_{\alpha,k})^2}R_{\alpha,k}^{j+k}\cos(\theta_{\alpha,k}(j-k))=0$$
    with $\mathcal X_{\alpha,k}(t_{\alpha,k})=R_{\alpha,k} e^{i\theta_{\alpha,k}}.$\\

One of the most important question is to see if the critical   value $t_{\alpha,k}$ of $\beta$ that realizes the minimum of $|\mathcal X_{\alpha,k}(\beta)|$  doesn't depend on $k$. It means that all functions attempt their minimums at the same "time".\\

For even and odd kernels, can we prove Corollary \ref{cor3} with $\eta=1$? Indeed, if we show that the zeros of $\mathcal I_{\alpha,\beta}(\xi)$ and $\mathcal J_{\alpha,\beta}(\xi)$ that converges to $\pm i$ are in $\mathbb D$ then we conclude the result. We note that this fact is confirmed numerically for some values of $\alpha$.

\section*{Annex: Numerical results}
All figures of this paper were produced using Python software. We give here the used code.\\
******************************************
\begin{verbatim}
from scipy import special
from sympy.abc import x, y, z
def A(beta, alpha):
 s=0
 for j in range(0, alpha+2):
 s=s+ ((1)/(j+beta)*special.binom(alpha+1,j)*(-x)**j)
 return s
import numpy
from cmath import *
from sympy.solvers import solve
import csv
DATA_PATH = '/content/drive/My Drive/graphes data/alpha7_7.csv'
i=1
with open(DATA_PATH, mode='w', newline='') as points_file:
 points_writer = csv.writer(points_file, delimiter=',')
 for beta in numpy.arange(10**(-6), 1, 10**(-2)):
 row = []
 for s in solve(A(beta, 6), x):
 sol = complex(s)
 row.append(-beta)
 row.append(sol.real)
 row.append(sol.imag)
 points_writer.writerow(row)
 print(i)
 i=i+1
 i=1
 for beta in numpy.arange(1-10**(-2), 1, 10**(-3)):
 row = []
 for s in solve(A(beta, 6), x):
 sol = complex(s)
 row.append(-beta)
 row.append(sol.real)
 row.append(sol.imag)
 points_writer.writerow(row)
 print(i)
 i=i+1
\end{verbatim}
**********************************************\\
In the following, we present some zeros sets of $G_{\alpha,\beta}$ for $\beta=-10^{-4}$ and $\alpha\in\{31,32,33,34,35,36,49,51,101\}$. The values of $\alpha$ and $\beta$ are chosen arbitrary just to see that there is no geometric stability of these zeros. It is possible that the geometric distribution of zeros of $G_{\alpha,\beta}$ depends on both $\alpha$ and $\beta$ in a complicate manner. It is also possible that if we see numerically  the geometric distribution of zeros of $G_{\alpha,\beta}$ for $\alpha$ large enough, then some new limit curve appear. However as a material problem, it was not possible for us to exceed the value $\alpha=101$.

\begin{figure}[h!]
   \includegraphics[scale=0.13]{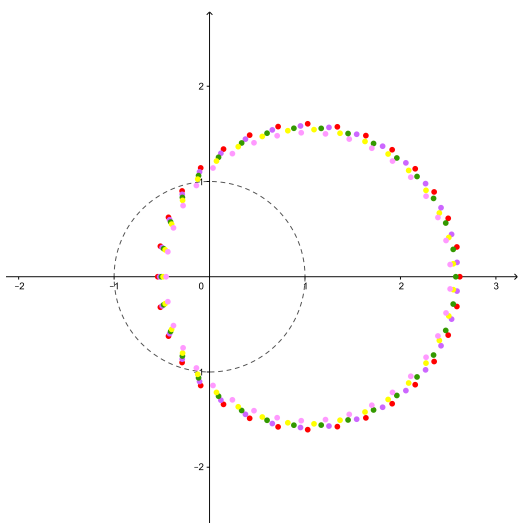}\\
   \caption{The Sets $\mathcal X_{\alpha,\bullet}(-10^{-4})$ for $\alpha\in\{31,32,33,34,36\}$}\label{fig4}
\end{figure}

\begin{figure}[h!]
   \begin{minipage}[b]{0.48\linewidth}
      \centering
      \includegraphics[scale=0.1]{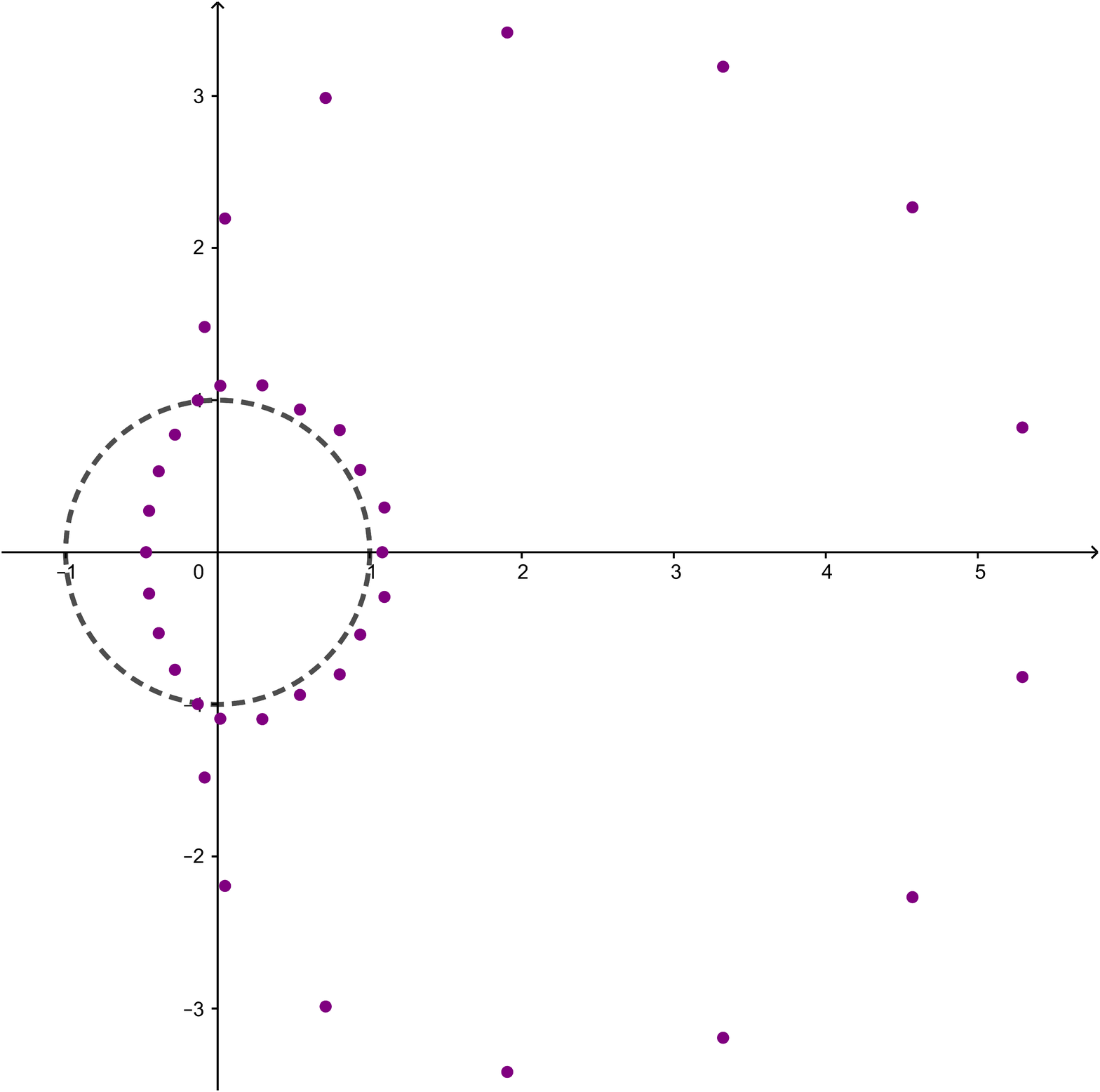}\\
   \end{minipage}\hfill
   \begin{minipage}[b]{0.48\linewidth}
      \centering
         \includegraphics[scale=0.125]{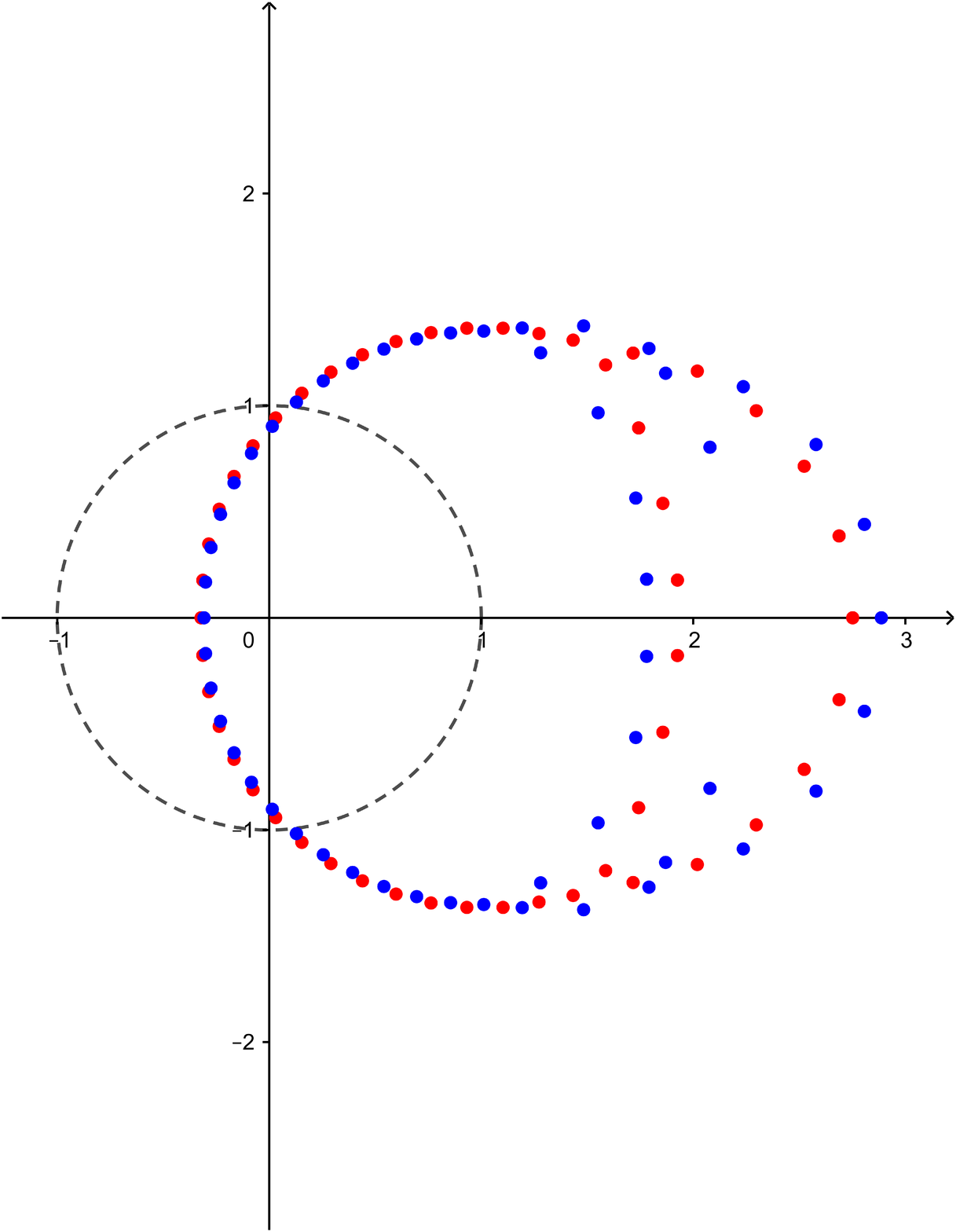}\\
   \end{minipage}
   \caption{The Sets $\mathcal X_{\alpha,\bullet}(-10^{-4})$ for $\alpha=35$ (in violet) at left and $\alpha=49$ (in red) and $\alpha=51$ (in blue) at right}\label{fig5}
\end{figure}

\begin{figure}
     \includegraphics[scale=0.12]{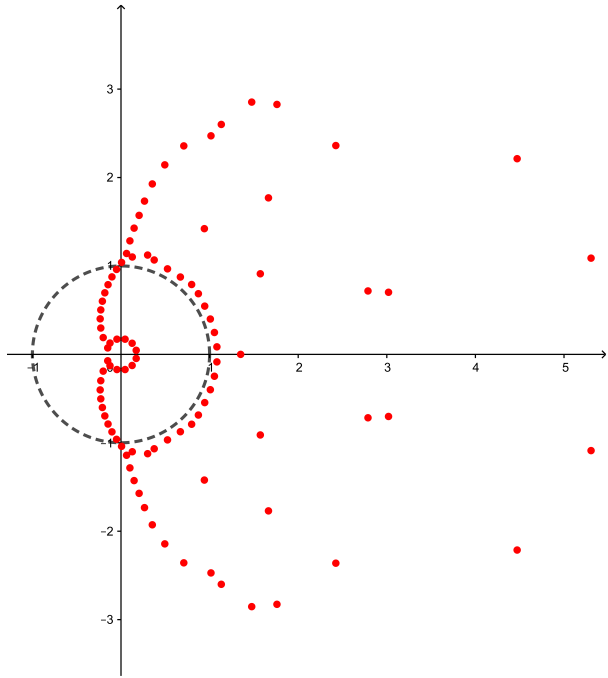}\\
   \caption{The Set $\mathcal X_{101,\bullet}(-10^{-4})$}\label{fig6}
\end{figure}
\end{document}